\documentclass{elsarticle}
\usepackage{amsthm}
\usepackage{amsmath}
\usepackage{amssymb}
\usepackage{enumerate}
\usepackage{color}
\usepackage[labelsep=quad,indention=10pt]{caption}
\usepackage[labelfont=bf,list=true]{subcaption}

\usepackage{tikz}

\makeatletter
\def\ps@pprintTitle{%
  \let\@oddhead\@empty
  \let\@evenhead\@empty
  \def\@oddfoot{\reset@font\hfil\thepage\hfil}
  \let\@evenfoot\@oddfoot
}
\makeatother

\newtheorem{theorem}{Theorem}[section]

\newtheorem{corollary}[theorem]{Corollary}

\newtheorem{lemma}[theorem]{Lemma}
\newtheorem{observation}[theorem]{Observation}

\theoremstyle{definition}
\newtheorem{definition}[theorem]{Definition}
\newtheorem{remark}[theorem]{Remark}
\def\N{\mathbb{N}}
\def\cS{\mathcal{S}}

\def\d{\mathbf{d}}
\def\G{\mathbb{G}}

\def\PP{\Pi}
\def\GG{\mathbb{G}}

\def\bd{\mathbf{d}}

\def\J{\mathcal{J}}

\def\cB{\mathcal{B}}

\newcommand\old[1]{}

\def\deg{d}
\begin{document}
\begin{frontmatter}
\title{Graph realizations constrained by skeleton graphs\tnoteref{ack}}
\author[renyi]{P\'eter L. Erd\H os\fnref{miklos}}
\author[ucd]{Stephen G. Hartke\fnref{hartke}}
\author[tud]{Leo van Iersel\fnref{lvi}}
\author[renyi]{Istv\'an Mikl\'os\fnref{miklos}}
\address[renyi]{Alfr\'ed R{\'e}nyi Institute, Re\'altanoda u 13-15 Budapest, 1053 Hungary\\
{\tt email:} $<$erdos.peter,miklos.istvan$>$@renyi.mta.hu}
\address[ucd]{Department of Mathematical and Statistical Sciences, University of Colorado Denver, Colorado, Denver, USA\\
{\tt email:} stephen.hartke@ucdenver.edu}
\address[tud]{Delft Institute of Applied Mathematics, Delft University of Technology, PO-box 5, 2600AA, Delft, Netherlands\\
{\tt email:} l.j.j.v.iersel@gmail.com}
\tnotetext[ack]{This research started when the 2nd and 3rd authors visited the MTA A. R\'enyi Institute of Mathematics, Budapest in the Fall of 2013. }
\fntext[hartke]{Partly supported by a U.S. Fulbright Scholar Fellowship and by a grant from the Simons Foundation (\#316262 to Stephen Hartke).}
\fntext[lvi]{Partly funded by the Netherlands Organisation for Scientific Research (NWO), including Veni grant 639.071.106 and Vidi grant 639.072.602 and by the 4TU Applied Mathematics Institute.}
\fntext[miklos]{PLE and IM were supported partly by the Hungarian National Research, Development and Innovation Office NKFIH, under the grants K 116769 and SNN 116095.}
\begin{abstract}
In 2008 Amanatidis, Green and Mihail introduced the {\em Joint Degree Matrix} (JDM) model to capture the fundamental difference in {\em assortativity} of networks in nature studied by the physical and life sciences and social networks studied in the social sciences.  In 2014 Czabarka proposed a direct generalization of the JDM model, the {\em Partition Adjacency Matrix} (PAM) model. In the PAM model the vertices have specified degrees, and the vertex set itself is partitioned into classes. For each pair of vertex classes the number of edges between the classes in a graph realization is prescribed.
In this paper we apply the new {\em skeleton graph} model to describe the same information as the PAM model. Our model is more convenient for handling problems with low number of partition classes or with special topological restrictions among the classes. We investigate two particular cases in detail: (i) when there are only two vertex classes and
(ii) when the skeleton graph contains at most one cycle.
\end{abstract}
\begin{keyword}
\MSC[2010]{05C07}\\
degree sequences; Joint Degree Matrix; Partition Adjacency Matrix; skeleton graph; forbidden edges; Tutte gadget; Edmonds's blossom algorithm
\end{keyword}
\end{frontmatter}

\section{Introduction}\label{sec:intro}
In the last fifteen years, the exponential development of network theory has raised the {\em practical} problem of realizing and sampling large graphs with given degree sequences. Finding a realization of a given degree sequence (among simple graphs or graphs with a given maximum number of parallel edges or/and loops) has long been shown to be an easy problem. Havel~\cite{havel} first solved the problem in 1957, and his algorithm was reinvented by Hakimi~\cite{hakimi} in 1962. An even older but less efficient way to find realizations can be derived from Tutte's $f$-factor results \cite{T52,T54}. Another method was due to Paul Erd\H{o}s and Gallai~\cite{EG60} in 1960, but the resulting algorithm was derived from Havel's approach.  All these methods lack the ability of generating all (or even a large number of different) realizations. The problem of determining if there exists a graph with given degree sequence and satisfying other specified conditions will be called in this paper the {\bf realization} problem.

In many situations, it is not feasible to generate all realizations, as the number of possible realizations can be exponential or larger in the length of the degree sequence. In this case, practical applications may require reasonable sampling methods of the ``typical'' realizations. A common approach is to use Markov Chain Monte Carlo methods, which require some simple operations that transform one realization into another, slightly different realization.
Additionally, it must be possible to transform any given realization into any other using these operations.
This particular subproblem of the sampling procedure will be called in this paper the {\bf connectivity} problem (also known as the irreducibility problem in the context of Markov Chain Monte Carlo processes for sampling random realizations).

In modern graph theory the first such manipulation was Havel's {\em swap} operation from \cite{havel}. (The terminology and notation used in this paper will be introduced in detail in Section~\ref{sec:pre}.) It is interesting to remark that his method was applied already by Petersen~\cite{pet} in 1891, who essentially showed that any realization of a given degree sequence can be transformed into any other realization of the same degree sequence by a series of such swap operations. In 1951 Senior \cite{S51} also used this approach to construct possible hydrocarbon molecules with given atomic composition.
For bipartite graphs it was done by Ryser~\cite{ryser} in 1963, and all these results have been invented again and again.

Recently a large number of real-world social and biological networks were studied in detail. One important distinction between these two types of networks lies in their overall structure: social networks typically have a few very high degree vertices and many low degree vertices with high {\em assortativity} (where a vertex is likely to be adjacent to vertices of similar degree), while biological networks are generally {\em disassortative} (in which low degree vertices tend to attach to those of high degree). It is easy to see that the {\em degree sequence} alone cannot capture these differences. There are several approaches to address this problem. See the paper of Stanton and Pinar~\cite{stanton} for a detailed description of the current state-of-the-art.

One way to address this problem is the {\bf joint degree matrix} model (or JDM for short). This model is more restrictive than the degree distribution, but it provides a way to enhance results based on degree distribution. It was introduced by  Amanatidis, Green and Mihail~\cite{agm08}  and Stanton and  Pinar~\cite{stanton}. In essence, the JDM specifies the exact {\em number} of edges between vertices of degrees $i$ and $j$. More precisely, a {\em joint degree matrix} $\J(G)=[\J_{ij}]$ of the graph $G=(V,E)$ is a $\Delta \times\Delta$ matrix ($\Delta$ is the maximum degree of $G$) where $\J_{ij}=|\{xy\in E(G): \,d(x)=i \text{ and } d(y)=j\}|$.
It is clear that the degree sequence of the graph is determined by its JDM:
\begin{equation}\label{eq:JDM}
\qquad (\text{the number of vertices of degree $i$}) = \frac{1}{i} \left (\J_{ii} + \sum_{\ell=1}^{\Delta} \J_{i \ell} \right ).
\end{equation}
The novelty of this definition is that values $\J_{i j}$ are exact numbers, and not expectations, like in earlier approaches, see for example \cite{dk06}.

The existence problem for the JDM model is not hard: already Patrinos and Hakimi~\cite{patrinos} presented in 1976 an Erd\H{o}s-Gallai-type theorem for joint degree matrices, essentially characterizing precisely those matrices which are the joint degree matrix for some graph, though using different terminology. Another proof for this result was given in \cite{agm08}, see also \cite{agm15}. Stanton and Pinar~\cite{stanton} gave a separate, constructive proof for this theorem, which builds a particular graph that has a given matrix as its JDM.  Czabarka, Dutle, Erd\H{o}s and Mikl\'os~\cite{JDM} presented a simpler proof using a construction algorithm that can create \emph{every} graph with a given JDM. See also \cite{GTM15}.

The connectivity problem for the JDM model proved to be more complicated. Stanton and Pinar~\cite{stanton} solved it for the space of all multigraph realizations. For simple graphs it was resolved affirmatively by Czabarka, Dutle, Erd\H{o}s and Mikl\'os~\cite{JDM}.

The JDM model suggests a more general restricted degree sequence problem: the {\bf partition adjacency matrix} model (or PAM for short). In this generality, it was introduced by \'E. Czabarka~\cite{PAM}. Let $\PP=(P_1, \cdots P_k)$ be a partition of the vertex set $V$ of the graph $G$.  Let $\bd$ be the degree sequence of $G$, and let $M$ be the following $k\times k$ matrix: if $i\ne j$ then the entry $M_{i,j}$ is the number of edges in the bipartite subgraph $G[P_i, P_j]$, while $M_{i,i}$ is the number of edges within the induced subgraph $G[P_i]$. The matrix $M$ is called the {\em partition adjacency matrix} of the graph $G$ for the partition $\PP$. Clearly the {\em PAM-problem} is: we are given a positive integer sequence on the partitioned $V$ and a matrix $M$ and we want to decide whether there is a graph $G$ with the given degree sequence and with the given PAM.

The joint degree  matrix determines the degree sequence itself by equation (\ref{eq:JDM}). Therefore the JDM is clearly an instance of the PAM-problem. The PAM problem in full generality is probably quite complicated: when we have, say, $\sqrt{|V|}$ partition classes, then the problem is conjectured to be NP-complete.

\bigskip\noindent
In this paper we will consider an auxiliary structure to describe the PAM problem. This is a more convenient description when there are only a small, say, linear number of items in $M$ which are not zero. This object also provides structural properties of the edges among the vertex partitions. This description is based on the notion of a {\em skeleton graph}, and it is described in detail in Section~\ref{sec:pre}.

We study two particular skeleton graphs in detail: the first one consists of two partition classes with edges inside the classes allowed, while in the second one each of its components contains at most one cycle while the classes have no edges inside. In both cases we show how to quickly construct graphical realizations in all feasible cases. We also consider whether the space of all realizations are connected. The answers in both cases are almost affirmative: the space is connected if we use swaps as well as an additional operation called \emph{double swaps}.

\section{Definitions and tools}\label{sec:pre}

Let $G=(V,E)$ be a \emph{simple} graph with vertex set $V=\left\{v_1,\dots,v_n\right\}$ and edge set $E$ (no multiple edges nor loops). Denote $\bd(G)=(d(v_1),\dots,d(v_n))$ its \emph{degree sequence}. This sequence is  not ordered in any way. $G$ is called a \emph{realization} of the previous sequence. A sequence $\bd=(d_1,\dots,d_n)$  of nonnegative integers is {\bf graphical} if it has at least one realization. The set of all realizations of a graphical degree sequence~$\d$ is denoted by $\GG(\bd)$.

We consider realizations of a degree sequence $\bd=(d_1,\dots,d_n)$ that are restricted by a ``skeleton'' graph~$\cS$ in the following way. We fix a partition $\Pi=\{U, W, \ldots\}$ of $V.$

A {\bf skeleton graph} is an edge-weighted graph ~$\cS= (\Pi;\cB,w)$ on the partition classes in $\Pi$ with possible loops. We will refer to a vertex of $\cS$ as a {\bf class} and the edges of the skeleton graph $\cS$ are referred to as {\bf bones}. We will use the following, very natural, notation: for $U, W \in \Pi$, if the pair $UW$ is a bone in~$\cB$, then $w(UW)$ is its weight (in $\cS$). If $U=W$ then it is a loop. If~$UW$ is not a bone in~$\cB$ then~$w(UW)=0$.Furthermore, $G[U,W]$ is the induced bipartite subgraph in $G$ while $G[U,U]=G[U]$ is the induced simple graph within $U.$ The graphs $G[U,W]$ are the \emph{component graphs} of $G.$ When the partition classes in $\Pi$ consist of all vertices with the same degrees, then all the introduced notions are equivalent to the notions used for the Joint Degree Matrix model in \cite{JDM}.

We will say that the realization $G$ of degree sequence $\d$ is \textbf{consistent} with the skeleton graph $\cS$ if
\begin{equation}\label{eq:skel}
\forall U,W \in \Pi \quad w(UW) = \big | E\left (G[U,W]\right )\!\big |.
\end{equation}
We say that the pair $(\bd, \cS)$ is {\bf graphical} if there exists a realization $G$ of $\bd$ which is consistent with $\cS$. See for example in Figure~\ref{fig:skeleton} a skeleton graph $\cS$ and a realization $G$ that is consistent with $\cS$.
\begin{figure}[t]
\centering
\begin{tikzpicture}[scale=.9]
\draw [line width=.3pt,dashed,fill=gray!50!white] (-.1,.5) rectangle (-.8,4.5);
\draw [line width=.3pt,dashed,fill=gray!50!white] (-.5,.5) rectangle (3.5,0);
\draw [line width=.3pt,dashed,fill=gray!50!white] (-.5,4.3) rectangle (3.5,4.8);
\draw [line width=.3pt,dashed,fill=gray!50!white] (3.2,.5) rectangle (3.8,4.5);
\draw [line width=.3pt,dashed,fill=gray!50!white] (-1,-.5)  circle (30pt);
\draw [line width=.3pt,dashed,fill=white] (-1,-.5)  circle (20pt);
\draw [line width=.3pt,dashed,fill=white] (-.45,.5)  circle (24pt);
\draw [line width=.3pt,dashed,fill=white] (-.45,4.5)  circle (24pt);
\draw [line width=.3pt,dashed,fill=white] (3.5,.5)  circle (24pt);
\draw [line width=.3pt,dashed,fill=white] (3.5,4.5)  circle (24pt);
\node (a) at (-.2,4.5)[label=above:$a$] {};
\node (b) at (-.6,4) [label=120:$b$] {};
\node (c) at (-.7,.7)[label=left:$c$] {};
\node (d) at (-.2,.2)[label=left:$d$] {};
\node (e) at (3.5,.5) [label=300:$e$]{};
\node (f) at (3.7,4)[label=above:$f$] {};
\node (g) at (3.3,4.5)[label=above:$g$] {};
\fill [color=black] (d)  circle (4pt);
\fill [color=black] (c) circle (4pt);
\fill [color=black] (b) circle (4pt);
\fill [color=black] (a) circle (4pt);
\fill [color=black] (g) circle (4pt);
\fill [color=black] (f) circle (4pt);
\fill [color=black] (e) circle (4pt);
\draw [line width=1pt] (c) -- (b);
\draw [line width=1pt] (c) -- (a);
\draw [line width=1pt] (d) -- (a);
\draw [line width=1pt] (c) .. controls (-3.8,-1.1) and (.4,-2.5).. (d);
\draw [line width=1pt] (a) -- (g);
\draw [line width=1pt] (e) -- (g);
\draw [line width=1pt] (e) -- (f);
\draw [line width=1pt] (e) -- (d);
\draw (-1,2.5) node {\small 3};
\draw (1.5,-.3) node {\small 1};
\draw (1.5,5.1) node {\small 1};
\draw (4.1,2.5) node {\small 2};
\draw (-2.3,0) node {\small 1};
\end{tikzpicture}
\caption{The dashed circles illustrate the classes of a skeleton graph~$\cS$, while its bones are indicated by gray dashed tubes. The dots and solid lines show a graph~$G$ that is consistent with~$\cS$. \label{fig:skeleton}}
\end{figure}
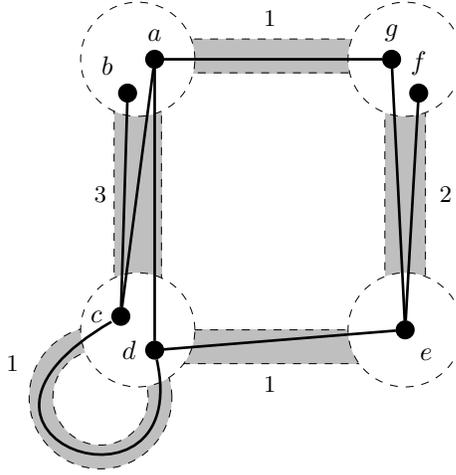
We will consider several different realizations of a degree sequence, even ones that are not consistent with the skeleton graph. However we will not consider realizations which have edges in a component graph without a corresponding bone. Therefore, we call a pair of vertices within a component graph with a corresponding bone a \emph{chord}, and each other pair of vertices a \emph{non-chord}. So a chord is a pair of vertices which may form an edge in a consistent realization.
A realization $G$ of a degree sequence $\bd$ is {\bf weakly consistent} with $\cS$ if all of the edges of $G$ are chords. Here property (\ref{eq:skel}) may not hold. Although the existence problem of weakly consistent realizations is not interesting in itself, we will use it as a tool. Before we discuss this, we recall some details about Tutte's $f$-factor theorem and its applications.

\medskip\noindent
In 1947 Tutte completely characterized the graphs with perfect matchings (see \cite{T47}). In 1952 he generalized this result for the so called $f$-factor problem. Two years later, in 1954, Tutte found the following brilliant way to reduce the problem of finding an $f$-factor in a given graph to finding a perfect matching (a 1-factor) in an auxiliary graph (\cite{T54}).
Let $G$ be a simple graph and let $f(v)$ be a non-negative integer for each $v\in V$. Then a subgraph $F$ of~$G$ in which each vertex~$v\in V$ has degree~$f(v)$ is an $f$-factor of $G.$ Any $f$-factor in $G$ can be represented as a perfect matching in the auxiliary graph $\mathbb{T}(G,f)$:
\begin{align}
V(\mathbb{T})&=\left\{ v_1,\ldots, v_{d(v)-f(v)} \big | v\in V(G)   \right \} \bigcup \left \{ e_v, e_u \big | vu=e\in E(G)    \right \} \label{eq:point}\\
E(\mathbb{T})&= \left\{v_ie_v \ \big | \ i=1,\ldots, d(v)-f(v); \ e=vu\in E(G) \right \} \bigcup \label{eq:edges}\\
& \qquad \bigcup \nonumber \left \{ e_ve_u \ \big | \ e=vu\in E(G) \right \}.
\end{align}
It is easy to see that there is a natural bijection between the $f$-factors in $G$ on one hand and the perfect matchings in $\mathbb{T}(G,f)$ on the other hand. More precisely, given the perfect matching $M$ in $\mathbb{T}(G,f)$, the requested subgraph in $G$ is
$$
\left\{e \in E \ \big | \ e=vu, e_ve_u \in  M  \right \}.
$$
In 1965, Edmonds described an effective algorithm to find a maximum matching in $G$ (\cite{E1965a}). Then, in the same year, he extended his approach for edge-weighted graphs (\cite{E1965b}): his blossom algorithm finds a maximum-weight perfect matching in strongly polynomial time. We will use this result extensively in this paper.
The classical existence problem for degree sequences can be easily solved by Tutte's $f$-factor theorem and Edmonds' algorithm: take as graph~$G$ the complete graph on $n$ vertices and for the function~$f$ the degree sequence. (See for example \cite{JSV} for an outstanding application.) However, this method is less efficient than the methods based on Havel's observation, and it can not be used to find all possible realizations (see for example \cite{Eg-tetel}).
Nevertheless, Edmonds' blossom algorithm is excellent to find weakly consistent realization to any skeleton graph:
\begin{observation}\label{th:weak}
For any degree sequence $\bd$ and any skeleton graph $\cS=(\Pi;\cB,w)$, one can decide in strongly polynomial time whether there exists a realization of~$\bd$ that is weakly consistent with~$\cS$.
\end{observation}
\begin{proof}
The graph $G$ in the $f$-factor problem consists of all chords defined by the skeleton graph, while the $f$-function is equal to the given degree sequence.
\end{proof}
\noindent
It is clear that the $f$-factor approach cannot directly find consistent realizations for any "reasonable"  skeleton graph problem: it has no control over the exact number of edges in the component graphs. We even cannot enforce that all bones contain at least one edge in the derived realization. We need additional ideas to find consistent realizations. For that end we will extensively use some restricted versions of Havel's swap operation.

\begin{definition}[{\bf unrestricted} / {\bf restricted} / {\bf $\cS$-preserving} swap operations]\ \\[-23pt]
\begin{enumerate}[{\rm (R1)}]
\item Let $G$ be a realization of the graphical sequence $\bd$, if $a,b,c$ and~$d$ are vertices of $G$ satisfying $ab,cd\in E$ and $bc,ad\notin E$, then the graph $G' =(V,E')$ with $E'=E\cup\{bc,ad\} \setminus \{ab,cd\}$ is another realization of~$\bd$. This \textbf{swap} operation, denoted by $ab,cd\Rightarrow bc,ad$, was introduced by Havel \cite{havel}. It is also known, for example, as {\it switch} or {\it rewiring} or {\it infusion} operation.
\item Let $\cS$ be a skeleton graph, and let $G$ be a realization of~$\bd$ that is weakly consistent with~$\cS$. If all vertex pairs in (R1) are chords, then $G'$ will be weakly consistent with $\cS.$ Then the operation is a {\bf restricted swap} operation.
\item If $G$ is consistent with the skeleton graph $\cS$ and $G'$ is also consistent with~$\cS$, then this operation is an {\bf $\cS$-preserving swap} operation.
\end{enumerate}
\end{definition}
\noindent
As we mentioned earlier already Petersen proved (\cite{pet}) that any realization of a given degree sequence can be transformed into another one by consecutive unrestricted swap operations. Havel's result gives a rather  crude algorithm to find such a swap sequence: the number of steps may be twice the number of edges in the worst case. It is very natural to ask what the minimum length of such a swap sequence is. This question was studied in details by Erd\H{o}s, Kir\'aly and Mikl\'os (\cite{distance}). Next we will summarize  the main findings of this paper:

\medskip\noindent {\bf Regular swap sequences}:
Let $G$ and $G'$ be two realizations of the degree sequence $\bd.$ The symmetric difference $\nabla=E(G) \triangle E(G')$ of their edges has a natural 2-coloration: an edge in $\nabla$ is {\bf red} or {\bf blue} depending on whether it belongs to $G$ or $G'$. Denote by $r(G,G')$ the number of red edges in the symmetric difference (which is of course also the number of blue edges).
\begin{lemma}\label{th:1}
Every vertex in $\nabla$ has an equal number of red and blue adjacent edges. Moreover, the symmetric difference $\nabla$ can be decomposed into even length alternating (with respect to the coloration) circuits (closed walks), where no circuit contains any vertex more than twice. \qed
\end{lemma}
\noindent
Consider now two realizations $G$ and $G'$ such that $\nabla$ is one alternating circuit, $C$.
\begin{lemma}\label{th:2}
There exists a sequence of consecutive swap operations transforming realization $G$ into $G'$ with the following properties: {\rm (i)} Along the process every swap is applied for vertex pairs belonging completely to $V(C)$. {\rm (ii)} If $G_1$ and $G_2$ are two consecutive realizations along the sequence then $r(G_1,G') - r(G_2,G') \in\{0,1,2\}$. {\rm {(iii)}} The length of this swap sequence is $r(G,G')-1.$ \qed
\end{lemma}
The described swap sequence is called a {\bf regular swap sequence}.

\begin{theorem}[Erd\H{o}s - Kir\'aly - Mikl\'os, \cite{distance}] \label{th:kiraly} \label{lem:shortest}
Let $G$ and $G'$ be arbitrary realizations of degree sequence $\bd.$ Every shortest possible swap sequence can be reordered such that this realigned sequence is identical with a series of subsequent regular swap sequences, corresponding to a circuit decomposition of the symmetric difference $\nabla.$ The length of this swap sequence is $r(G,G') - $ the maximum possible number of circuits in a decomposition, and hence is at most $r(G,G')-1$.
\end{theorem}
\noindent With some lack of precision we also call the swap sequence described above as a {\bf regular swap sequence}, and a shortest regular swap sequence, respectively.
\noindent Two further useful observations:
\begin{remark}\label{th:extend}\   \\[-17pt]
\begin{enumerate}[{\rm (i)}]
\item Any particular alternating circuit in $\nabla$ can be extended into a complete decomposition of $\nabla.$\\[-17pt]
\item If an arbitrary swap sequence transforms $G$ into $G'$ then the {\em inverse} swap sequence transforms $G'$ into $G$ (here we do not define the notion inverse, because it is self-evident). \\[-17pt]
\end{enumerate}
\end{remark}
\noindent The following is easily verified.
\begin{observation}
In an $\cS$-consistent realization $G$, a swap $ab,cd\Rightarrow bc,ad$ is $\cS$-preserving if and only if $(a$ and $c)$ or $(b$ and $d)$ are in the same class of $\cS$.
\end{observation}
\noindent
One of our motivating questions is whether the space of realizations of a graphical sequence consistent with a given skeleton graph is connected using swaps. In the classic cases this is true (see, for example, Theorem \ref{th:kiraly}) as well as in the Joint Degree Matrix case (see \cite{JDM}). However, as we will see later on, there is not so neat answer for the skeleton graph problems. The obvious reason for this is that Theorem \ref{th:kiraly} does not apply for this case, since a regular swap sequence does not necessarily use $\cS$-preserving swap operations.

\section{Graph realizations with a given number of edges crossing a given bipartition}
We start our investigations with one of the most simple skeleton graphs. Let $V=\{v_1,\ldots ,v_n\}$ be a vertex set, $\bd$ a degree sequence, $\Pi^2 = (U,W)$ a partition of~$V$, $k\in\N$ and let $\cS(k)= (\Pi^2; \cB,w)$ be a skeleton graph with two vertices, a weight-$k$ $UW$ bone and two loops, whose weights are completely defined by $\bd$ and $k.$ The edges with end vertices in both classes are called the {\bf crossing} edges.

The conventional description is the following: Given the degree sequence $\bd$, a bipartition of the vertex set and a natural number~$k$, decide if there exists a graph that realizes the given degree sequence and has precisely~$k$ crossing edges.

When $k$ is equal to the total number of edges, i.e. both classes induce the empty graph, then this coincides with the usual bipartite degree sequence problem.

When $W$ has no inner edges, then no direct greedy method (like Havel's lemma) is known to construct a consistent realization. The reason for that is simple: we just do not know the $U$-side of the bipartite degree sequence in the component graph $G[U,W]$. However, the Tutte - Edmonds method provides an effective solution for the existence problem.

\subsection{Existence}\label{subsec:exist}
\noindent It is clear that any realization of $\bd$ is automatically weakly consistent. In this subsection we consider the existence problem for consistent realizations.
\begin{theorem}\label{thm:2part}
We can decide in polynomial time whether there exists a realization of~$\d$ that is consistent with $\cS(k)$.
\end{theorem}
The number of crossing edges in a realization~$G$ is denoted $\epsilon(G)$. The set of all realizations of degree sequence $\bd$ with~$\ell$ crossing edges is denoted~$\GG_{\ell}(\bd)$.

Let $G\in \GG_\ell(\bd)$ and let $G'$ be the realization derived from $G$ by the swap operation $ac, bd \Rightarrow bc, ad$. By simple case analysis it is easy to see that
\begin{equation}\label{eq:par}
\left | \epsilon(G) - \epsilon(G')\right | \in\{0,2\}.
\end{equation}
\begin{lemma}\label{th:parity}
Let $G$ and $G'$ be realizations of $\bd$ with $\epsilon(G) < \epsilon(G')$. Then
\begin{enumerate}[{\rm (i)}]
\item $\epsilon(G) \equiv \epsilon(G') \pmod{2}$;
\item for all $\ell\in \{\epsilon(G), \epsilon(G)+2,\ldots, \epsilon(G')\}$, there exists a realization $G'' \in\GG_\ell(\bd).$
\end{enumerate}
\end{lemma}
\begin{proof}
By Theorem \ref{th:kiraly} there is a regular swap sequence turning $G$ into $G'$. By equation~(\ref{eq:par}), the realizations in this sequence either all have an even number of crossing edges or all have an odd number of crossing edges. Moreover, this sequence hits a realization $G''$ with $\epsilon(G'')=\ell$ for any $\epsilon(G) \leq \ell \leq \epsilon(G')$ with $\ell=\epsilon(G) \pmod{2}$.
\end{proof}

Denote by $\epsilon_m(\bd)$ the {\em minimum} value of $\ell$ such that $\GG_\ell$ is not empty, and similarly denote by $\epsilon_M(\bd)$ the {\em maximum} value. By Lemma~\ref{th:parity} the sets
\begin{equation}\label{eq:non-empty}
\GG_{\epsilon_m(\bd)}(\bd), \GG_{\epsilon_m(\bd)+2}(\bd), \ldots, \GG_{\epsilon_M(\bd)}(\bd)
\end{equation}
of all weakly consistent realizations are not empty while all other sets $\GG_\ell(\bd)$ are empty.

Edmonds' blossom algorithm \cite{E1965b} applied to the Tutte gadget $\mathbb{T} (K_n,\bd)$ can easily find the minimum and maximum values $\epsilon_m(\bd)$ and $\epsilon_M(\bd)$ together with the corresponding realizations consistent with $\cS(\epsilon_m(\bd))$ and $\cS(\epsilon_M(\bd))$, respectively.

We first assign weight $0$ to all crossing edges in $\mathbb{T}(K_n,\bd)$, while to all other edges we assign weight $1$. Then any maximum weight perfect matching corresponds to a realization $G$ which contains the minimum possible number of crossing edges. Therefore $\epsilon_m(\bd) = \epsilon(G)$. To determine $\epsilon_M$ we  just reverse the edge weights. So the solution of the maximum weight matching problem for this graph provides the value $\epsilon_M(\bd)$ and a realization $G\in \GG_{\epsilon_M(\bd)}$.

The formula (\ref{eq:non-empty}) gives all other possible $\epsilon$ values while the proof of Lemma \ref{th:parity} (ii) describes the way to find realizations with  specific $\epsilon$-values. \qed

\medskip\noindent This proof seems to be easy and straightforward. However, the situation is more complicated. First of all, as far as the authors are aware, this is the very first solved degree sequence type problem without some direct, greedy type solution. Secondly, it is somewhat unusual to solve an existence problem, where the two extremal solutions can be found directly, while the solutions inbetween can be inferred only indirectly from the extremal solutions.
Finally, we think that it is a small miracle that this proof works at all. Consider the following, slightly different question: the vertices, equipped by a degree sequence $\bd$, are partitioned into three classes: $U, W, Z.$ We are looking for a realization $G$ of $\bd$ which has exactly $k$ edges between $U$ and $W.$ There is no any other restriction. (Of course this problem does not belong to the skeleton graph problem class, but the difference is tiny.) For this problem, Lemma \ref{th:parity} does not hold. Therefore, finding realizations with minimum $m$ and maximum $M$ number of crossing edges does not help to solve the problem if $m < k < M$. Actually, this problem seems to be quite hard, and the authors were not able to provide a polynomial-time algorithm solving it despite of serious efforts.
\subsection{Connectivity}
\noindent
Now consider two realizations $G,G'$ of degree sequence~$\bd,$ consistent with the skeleton graph~$\cS(k)=(\Pi^2;\cB,w)$. Is it true that in such a situation there always exists a sequence of $\cS$-preserving swaps that turn~$G$ into~$G'$? The following counter example shows that the answer to the above question is ``no''.

\begin{theorem}\label{thm:counter}
There exist graphs~$G,G'$ that have the same vertex set~$V$, the same degree sequence~$\bd$ and that are both consistent with the skeleton graph $\cS(k)=(\Pi^2; \cB,w)$ described above, such that there exists no sequence of $\cS$-preserving swaps that turn~$G$ into~$G'$.
\end{theorem}

\begin{proof}
We construct an example as follows. Let $\bd=(6,6,3,3,3,3,1,1)$ be the degrees of vertices~$(u_3,w_7,u_2,w_6,u_1,w_5,u_0,w_4)$. Let $\cS$ be the skeleton graph with $\Pi=\{U,W\}$ with $U=\{u_0,u_1,u_2,u_3\}, W=\{w_4,w_5,w_6,w_7\}$ and with bone $UW$ with weight $k=7$. Then both vertex classes must contain 3 edges. In Figure \ref{fig:counter} we show all realizations of this degree sequence that are consistent with~$\cS$. There is a single $\cS$-preserving swap that turns $G_1$ into $G_2$, namely the swap $u_2w_5,u_1w_6\Rightarrow u_1w_5,u_2w_6$. However, there is no $\cS$-preserving swap that turns~$G_1$ into~$G_3$ or~$G_2$ into~$G_3$. (See  Figure \ref{fig:counter1} where the red and blue chords can be found in only one realization, while the black chords are in both.) Hence, there is no sequence of $\cS$-preserving swaps that turns~$G_1$ or~$G_2$ into~$G_3$.
\begin{figure}[h]
\begin{minipage}[b]{.33\linewidth}
\centering
\begin{tikzpicture}[scale=.8]
\draw[color=black,fill=black] (0,0) node (u3) [label=below:{\phantom{$u_3$}}] {} circle (2pt);
\draw[color=black,fill=black] (-.5,1) node (u2) {} circle (2pt);
\draw[color=black,fill=black] (-.5,2) node (u1) {} circle (2pt);
\draw[color=black,fill=black] (0,3) node (u0) {} circle (2pt);
\draw[color=black,fill=black] (2,0) node (u7) {} circle (2pt);
\draw[color=black,fill=black] (2.5,1) node (u6) {} circle (2pt);
\draw[color=black,fill=black] (2.5,2) node (u5) {} circle (2pt);
\draw[color=black,fill=black] (2,3) node (u4) {} circle (2pt);
\draw [line width=.5pt] (u3) -- (u2);
\draw [line width=.5pt] (u3) -- (u1);
\draw [line width=.5pt] (u3) -- (u0);
\draw [line width=.5pt] (u3) -- (u5);
\draw [line width=.5pt] (u3) -- (u6);
\draw [line width=.5pt] (u3) -- (u7);
\draw [line width=.5pt] (u7) -- (u2);
\draw [line width=.5pt] (u7) -- (u1);
\draw [line width=.5pt] (u7) -- (u4);
\draw [line width=.5pt] (u7) -- (u5);
\draw [line width=.5pt] (u7) -- (u6);
\draw [line width=.5pt] (u1) -- (u6);
\draw [line width=.5pt] (u2) -- (u5);
\draw [dashed] (1,-.5) -- (1,3.5);
\draw (0,4) node {$U$};
\draw (2,4) node {$W$};
\end{tikzpicture}
\subcaption{$G_1$}\label{fig:c1a}
\end{minipage}%
\begin{minipage}[b]{.33\linewidth}
\centering
\begin{tikzpicture}[scale=.8]
\draw[color=black,fill=black] (0,0) node (u3) [label=below:$u_3$] {} circle (2pt);
\draw[color=black,fill=black] (-.5,1) node (u2) [label=left:$u_2$] {} circle (2pt);
\draw[color=black,fill=black] (-.5,2) node (u1) [label=left:$u_1$] {} circle (2pt);
\draw[color=black,fill=black] (0,3) node (u0) [label=above:$u_0$] {} circle (2pt);
\draw[color=black,fill=black] (2,0) node (u7) [label=below:$w_7$] {} circle (2pt);
\draw[color=black,fill=black] (2.5,1) node (u6) [label=right:$w_6$] {} circle (2pt);
\draw[color=black,fill=black] (2.5,2) node (u5) [label=right:$w_5$] {} circle (2pt);
\draw[color=black,fill=black] (2,3) node (u4) [label=above:$w_4$] {} circle (2pt);
\draw [line width=.5pt] (u3) -- (u2);
\draw [line width=.5pt] (u3) -- (u1);
\draw [line width=.5pt] (u3) -- (u0);
\draw [line width=.5pt] (u3) -- (u5);
\draw [line width=.5pt] (u3) -- (u6);
\draw [line width=.5pt] (u3) -- (u7);
\draw [line width=.5pt] (u7) -- (u2);
\draw [line width=.5pt] (u7) -- (u1);
\draw [line width=.5pt] (u7) -- (u4);
\draw [line width=.5pt] (u7) -- (u5);
\draw [line width=.5pt] (u7) -- (u6);
\draw [line width=.5pt] (u1) -- (u5);
\draw [line width=.5pt] (u2) -- (u6);
\draw [dashed] (1,-.5) -- (1,3.5);
\end{tikzpicture}
\subcaption{$G_2$}\label{fig:c1b}
\end{minipage}%
\begin{minipage}[b]{.33\linewidth}
\centering
\begin{tikzpicture}[scale=.8]
\draw[color=black,fill=black] (0,0) node (u3) [label=below:{\phantom{$u_3$}}] {} circle (2pt);
\draw[color=black,fill=black] (-.5,1) node (u2) {} circle (2pt);
\draw[color=black,fill=black] (-.5,2) node (u1) {} circle (2pt);
\draw[color=black,fill=black] (0,3) node (u0) {} circle (2pt);
\draw[color=black,fill=black] (2,0) node (u7) {} circle (2pt);
\draw[color=black,fill=black] (2.5,1) node (u6) {} circle (2pt);
\draw[color=black,fill=black] (2.5,2) node (u5) {} circle (2pt);
\draw[color=black,fill=black] (2,3) node (u4) {} circle (2pt);
\draw [line width=.5pt] (u3) -- (u2);
\draw [line width=.5pt] (u3) -- (u1);
\draw [line width=.5pt] (u3) -- (u4);
\draw [line width=.5pt] (u3) -- (u5);
\draw [line width=.5pt] (u3) -- (u6);
\draw [line width=.5pt] (u3) -- (u7);
\draw [line width=.5pt] (u7) -- (u2);
\draw [line width=.5pt] (u7) -- (u1);
\draw [line width=.5pt] (u7) -- (u0);
\draw [line width=.5pt] (u7) -- (u5);
\draw [line width=.5pt] (u7) -- (u6);
\draw [line width=.5pt] (u1) -- (u2);
\draw [line width=.5pt] (u6) -- (u5);
\draw [dashed] (1,-.5) -- (1,3.5);
\draw (0,4) node {$U$};
\draw (2,4) node {$W$};
\end{tikzpicture}
\subcaption{$G_3$}\label{fig:c1c}
\end{minipage}%
\caption{There is no sequence of $\cS$-preserving swaps that turns~$G_1$ or~$G_2$ into~$G_3$.
}
\label{fig:counter}
\end{figure}
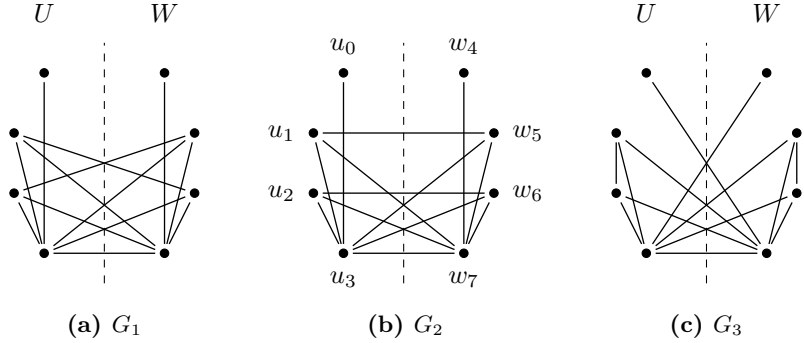
\end{proof}

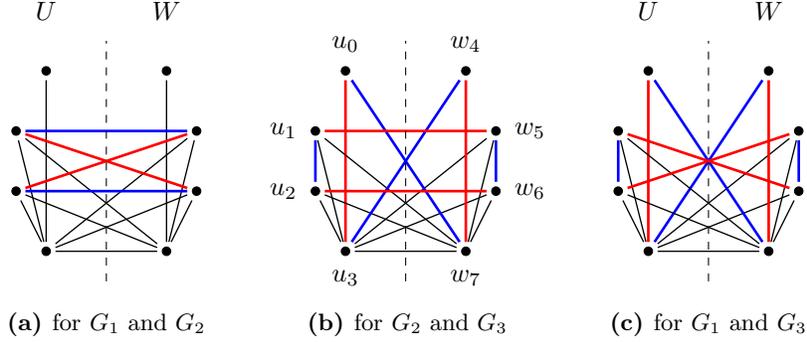
\begin{figure}[h]
\begin{minipage}[b]{.33\linewidth}
\centering
\begin{tikzpicture}[scale=.8]
\draw[color=black,fill=black] (0,0) node (u3) [label=below:{\phantom{$u_3$}}] {} circle (2pt);
\draw[color=black,fill=black] (-.5,1) node (u2) {} circle (2pt);
\draw[color=black,fill=black] (-.5,2) node (u1) {} circle (2pt);
\draw[color=black,fill=black] (0,3) node (u0) {} circle (2pt);
\draw[color=black,fill=black] (2,0) node (u7) {} circle (2pt);
\draw[color=black,fill=black] (2.5,1) node (u6) {} circle (2pt);
\draw[color=black,fill=black] (2.5,2) node (u5) {} circle (2pt);
\draw[color=black,fill=black] (2,3) node (u4) {} circle (2pt);
\draw [line width=.5pt] (u3) -- (u2);
\draw [line width=.5pt] (u3) -- (u1);
\draw [line width=.5pt] (u3) -- (u0);
\draw [line width=.5pt] (u3) -- (u5);
\draw [line width=.5pt] (u3) -- (u6);
\draw [line width=.5pt] (u3) -- (u7);
\draw [line width=.5pt] (u7) -- (u2);
\draw [line width=.5pt] (u7) -- (u1);
\draw [line width=.5pt] (u7) -- (u4);
\draw [line width=.5pt] (u7) -- (u5);
\draw [line width=.5pt] (u7) -- (u6);
\draw [line width=1pt,red] (u1) -- (u6);
\draw [line width=1pt,red] (u2) -- (u5);
\draw [line width=1pt,blue] (u1) -- (u5);
\draw [line width=1pt,blue] (u2) -- (u6);
\draw [dashed] (1,-.5) -- (1,3.5);
\draw [line width=.5pt] (0,4) node {$U$};
\draw [line width=.5pt] (2,4) node {$W$};
\end{tikzpicture}
\subcaption{for $G_1$ and $G_2$}\label{fig:c2a}
\end{minipage}%
\begin{minipage}[b]{.33\linewidth}
\centering
\begin{tikzpicture}[scale=.8]
\draw[color=black,fill=black] (0,0) node (u3) [label=below:$u_3$] {} circle (2pt);
\draw[color=black,fill=black] (-.5,1) node (u2) [label=left:$u_2$] {} circle (2pt);
\draw[color=black,fill=black] (-.5,2) node (u1) [label=left:$u_1$] {} circle (2pt);
\draw[color=black,fill=black] (0,3) node (u0) [label=above:$u_0$] {} circle (2pt);
\draw[color=black,fill=black] (2,0) node (u7) [label=below:$w_7$] {} circle (2pt);
\draw[color=black,fill=black] (2.5,1) node (u6) [label=right:$w_6$] {} circle (2pt);
\draw[color=black,fill=black] (2.5,2) node (u5) [label=right:$w_5$] {} circle (2pt);
\draw[color=black,fill=black] (2,3) node (u4) [label=above:$w_4$] {} circle (2pt);
\draw [line width=.5pt] (u3) -- (u2);
\draw [line width=.5pt] (u3) -- (u1);
\draw [line width=1pt,red] (u3) -- (u0);
\draw [line width=.5pt] (u3) -- (u5);
\draw [line width=.5pt] (u3) -- (u6);
\draw [line width=.5pt] (u3) -- (u7);
\draw [line width=.5pt] (u7) -- (u2);
\draw [line width=.5pt] (u7) -- (u1);
\draw [line width=1pt,red] (u7) -- (u4);
\draw [line width=1pt,blue] (u0) -- (u7);
\draw [line width=1pt,blue] (u3) -- (u4);
\draw [line width=.5pt] (u7) -- (u5);
\draw [line width=.5pt] (u7) -- (u6);
\draw [line width=1pt,red] (u1) -- (u5);
\draw [line width=1pt,red] (u2) -- (u6);
\draw [line width=1pt,blue] (u1) -- (u2);
\draw [line width=1pt,blue] (u6) -- (u5);
\draw [dashed] (1,-.5) -- (1,3.5);
\end{tikzpicture}
\subcaption{for $G_2$ and $G_3$}\label{fig:c2b}
\end{minipage}%
\begin{minipage}[b]{.33\linewidth}
\centering
\begin{tikzpicture}[scale=.8]
\draw[color=black,fill=black] (0,0) node (u3) [label=below:{\phantom{$u_3$}}] {} circle (2pt);
\draw[color=black,fill=black] (-.5,1) node (u2) {} circle (2pt);
\draw[color=black,fill=black] (-.5,2) node (u1) {} circle (2pt);
\draw[color=black,fill=black] (0,3) node (u0) {} circle (2pt);
\draw[color=black,fill=black] (2,0) node (u7) {} circle (2pt);
\draw[color=black,fill=black] (2.5,1) node (u6) {} circle (2pt);
\draw[color=black,fill=black] (2.5,2) node (u5) {} circle (2pt);
\draw[color=black,fill=black] (2,3) node (u4) {} circle (2pt);
\draw [line width=.5pt] (u3) -- (u2);
\draw [line width=.5pt] (u3) -- (u1);
\draw [line width=1pt,blue] (u3) -- (u4);
\draw [line width=1pt,red] (u3) -- (u0);
\draw [line width=.5pt] (u3) -- (u5);
\draw [line width=.5pt] (u3) -- (u6);
\draw [line width=.5pt] (u3) -- (u7);
\draw [line width=.5pt] (u7) -- (u2);
\draw [line width=.5pt] (u7) -- (u1);
\draw [line width=1pt,blue] (u7) -- (u0);
\draw [line width=1pt,red] (u7) -- (u4);
\draw [line width=.5pt] (u7) -- (u5);
\draw [line width=.5pt] (u7) -- (u6);
\draw [line width=1pt,blue] (u1) -- (u2);
\draw [line width=1pt,blue] (u6) -- (u5);
\draw [line width=1pt,red] (u5) -- (u2);
\draw [line width=1pt,red] (u1) -- (u6);
\draw [dashed] (1,-.5) -- (1,3.5);
\draw (0,4) node {$U$};
\draw (2,4) node {$W$};
\end{tikzpicture}
\subcaption{for $G_1$ and $G_3$}\label{fig:c2c}
\end{minipage}%
\caption{The red and blue chords come from one realization, while the black ones come from both. The realizations themselves come from Figure~\ref{fig:counter}.}
\label{fig:counter1}
\end{figure}
\noindent Motivated by Theorem~\ref{thm:counter}, we define a {\bf double swap} as the simultaneous application of two disjoint swaps, i.e., if $a,b,c,d,e,f,g,h$ are eight distinct vertices of graph~$G=(V,E)$ such that~$ab,cd,ef,gh\in E$ and $bc,ad,fg,eh\notin E$, then the result of applying the double swap $ab,cd,ef,gh \Rightarrow bc,ad,fg,eh$ is the graph $G'=(V,E\cup\{ ab,cd,ef,gh\}\setminus\{bc,ad,fg,eh\})$. A double swap is $\cS$-preserving, for some skeleton graph~$\cS$, if~$G'$ is consistent with~$\cS$.

\begin{theorem}\label{thm:doubleswaps}
For every two graphs~$G,G'$ that have the same vertex set~$V$, the same degree sequence~$\bd$ and that are both consistent with the same skeleton graph~$\cS(k)=(\Pi^2;\cB,w)$, there exists a sequence of $\cS$-preserving swaps and double swaps that turn~$G$ into~$G'$.
\end{theorem}

\noindent {\bf Proof of Theorem \ref{thm:doubleswaps}}.
The proof is by induction on~$r(G,G')$. Let's recall that this is the number of red edges in $G\triangle G'$ which is $\frac{1}{2}|E(G) \triangle E(G')|.$ If $r(G,G')=0$ then~$G=G'$ and we are done. Moreover, if $r(G,G')=2$ then we can turn~$G$ into~$G'$ by applying the single swap consisting of the red and blue edges of~$\nabla$ and we are again done. (It is easy to see that $r(G,G')$ cannot be 1.)  Hence, assume $r(G,G')\geq 3$. In a sequence of lemmas we will show that we can always find an $\cS$-preserving swap and/or double swap sequence which reduces $r(G,G')$. In each lemma we identify an alternating circuit in $\nabla$ s.t. the regular swap subsequence processing this circuit will apply only $\cS$-preserving swaps, and will hence satisfy the conditions described in Theorem~\ref{thm:doubleswaps}.
\begin{lemma}
If one can find an alternating circuit  $C\subset \nabla$ completely within a vertex class, say, $U$, then the regular swap subsequence processing $C$ (see Theorem \ref{th:kiraly}) provides $\cS$-preserving swaps at each step.
\end{lemma}
\begin{proof}
Indeed, the regular swap sequence can start with that particular circuit---see Lemma~\ref{th:extend}(i)---and the process uses chords of this circuit, therefore all swap operations happen inside class $U$ and hence all swaps are $\cS$-preserving.
\end{proof}
From now on we assume that there exists no alternating circuit completely within any vertex class.
To make the references easier for different type of chords in the realizations, we introduce two further notions: we say that a chord is {\bf black} if it is an edge in both $G$ and $G'$. It is {\bf white} if it is an edge neither in $G$ nor in $G'$.
\begin{lemma}\label{clm:no_touching_red_and_blue}
Assume that there exists a red crossing edge and a blue crossing edge that share an endpoint. Then we can identify an alternating circuit $C$ such that the regular swap subsequence which processes $C$ provides only $\cS$-preserving swaps.
\end{lemma}
\begin{figure}[h]
\centering
\begin{tikzpicture}[scale=1.2]
\draw[color=black,fill=black] (0,0) node (u) [label=above:$u$] {} circle (2pt);
\draw [color=black,fill=black] (-1,0) node (a) [label=above:$a$] {} circle (1pt);
\draw[color=black,fill=black] (0,-1) node (x) [label=below:$x$] {} circle (2pt);
\draw[color=black,fill=black] (1,-.5) node (v) [label=right:$v$] {}   circle (2pt);
\draw [color=black,fill=black] (-1,-1) node (b) [label=below:$b$] {} circle (1pt);
\draw [line width=1.5pt,red] (u) -- (v) node[pos=0.3,above,color=black] {$e_1$};
\draw [line width=1.5pt,blue] (x) -- (v) node[pos=0.3,below,color=black] {$e_2$};
\draw [line width=1pt,red] (x) -- (b);
\draw [line width=1pt,blue] (u) -- (a);
\draw [line width=.5pt,red,dashed] (u) -- (b);
\draw [line width=.5pt,blue,dashed] (x) -- (a);
\draw [dashed] (.55,-1.5) -- (.55,1);
\draw (0,1) node {$U$};
\draw (1,1) node {$W$};
\end{tikzpicture}
\caption{A red crossing chord and a blue crossing chord that share an endpoint. The red chord is an edge in $G$ but not in $G'$. Analogously the blue chord is an edge in $G'$ but not in $G.$ Either $au$ is blue but $ax$ is not blue, or $bx$ is red but $bu$ is not red.}
\label{fig:simp1}
\end{figure}
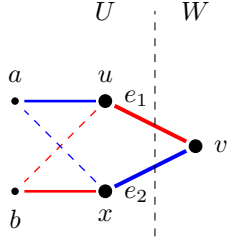

\begin{proof}
Suppose that $e_1=uv$ is a red crossing edge and $e_2=xv$ is blue crossing edge, where $u,x\in U$ and $v \in W$; see Figure~\ref{fig:simp1}.
We first show that there is a vertex $a \in V(G)\setminus\{u,v,x\}$ such that $au$ is blue but $ax$ is not blue, or there is a vertex $b \in V(G)\setminus\{u,v,x\}$ such that $bx$ is red but $bu$ is not red. If no such vertex $a$ exists, then $\deg_b(u)<\deg_b(x)$, and if no such vertex $b$ exists, then $\deg_r(x)<\deg_r(u)$. But since $\deg_r(u)=\deg_b(u)$ and $\deg_r(x)=\deg_b(x)$ we have $\deg_r(u)=\deg_b(u)<\deg_b(x)=\deg_r(x)$, which contradicts the previous sentence.

Hence such vertex $a$ or vertex $b$ exists. Without loss of generality assume that there exists a vertex $a$ such that $au$ is blue but $ax$ is not blue.

If $ax$ is red or black, then we apply the $\cS$-preserving swap $au,uv,xv,ax$ to $G$ to obtain $G^*,$  a consistent realization. Since $r(G^*,G') < r(G,G')$ therefore the induction applies.

\medskip\noindent
If $ax$ is white, then we apply the  swap $au, xv \Rightarrow ax, uv$ for realization $G$. (Here we apply Lemma \ref{th:extend} (ii).)  For the derived realization $G^*$ the number of red edges within $E(G) \bigtriangleup E(G^*)$ is smaller than $r(G,G')$ and induction applies.
\end{proof}
From now on we assume that there exist no adjacent red and blue crossing edges.
\begin{lemma}\label{clm:no_mixed_circuit}
Assume there exists a red-blue alternating trail $T$ where the first chord of $T$ is a red crossing one, the last chord is a blue crossing one, and the other chords of $T$ are not crossing ones. (Then $T$ consists of at least four chords.) Then we can identify an $\cS$ consistent realization $G^*$ with the property, that $|E(G)\triangle E(G^*)| < 2r(G,G')$ and $|E(G^*)\triangle E(G')|< 2 r(G,G')$, therefore the inductive hypothesis applies.
\end{lemma}

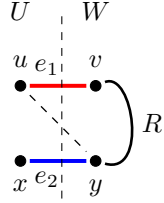
\begin{figure}[h]
\centering
\begin{tikzpicture}[scale=1]
\draw[color=black,fill=black] (0,0) node (u) [label=above:$u$] {} circle (2pt);
\draw[color=black,fill=black] (0,-1) node (x) [label=below:$x$] {} circle (2pt);
\draw[color=black,fill=black] (1,0) node (v) [label=above:$v$] {} circle (2pt);
\draw[color=black,fill=black] (1,-1) node (y) [label=below:$y$] {} circle (2pt);

\draw [line width=.5pt,dashed] (u) -- (y);
\draw [color=black,line width=1pt] (v) .. controls (1.5,.2) and (1.7,-1.2) .. (y) node [pos=0.5,right] {$R$};
\draw [line width=1.5pt,red] (u) -- (v) node[pos=0.3,above,color=black] {$e_1$};
\draw [line width=1.5pt,blue] (x) -- (y) node[pos=0.3,below,color=black] {$e_2$};

\draw [dashed] (.55,-1.5) -- (.55,1);
\draw (0,1) node {$U$};
\draw (1,1) node {$W$};
\end{tikzpicture}
\caption{A red crossing chord and a blue crossing chord connected by a red-blue alternating chord trail in~$G$. Recall, that a red chord is an actual edge in $G$ but a non-edge in $G'$ and a blue chord is a non-edge in $G$ but an edge in $G'$.} \label{fig:base}
\end{figure}

\begin{proof}
Suppose that such an alternating trail $T$ exists.  Let $e_1=uv$ be the red crossing edge in $T$ and $e_2=xy$ be the blue crossing edge, where $u,x\in U$, $v,y\in W$, and the other edges of~$T$ lie within $W$.  Let $R$ be the subtrail of $T$ from $v$ to $y$ formed from $T$ by removing $e_1$ and $e_2$, then $R$ contains at least two chords.  See Figure~\ref{fig:base} for a picture.

Consider the pair $uy$.  Since Lemma~\ref{clm:no_touching_red_and_blue} does not apply, $uy$ is a black or white chord. If $uy$ is black, then denote by $C$ the chord circuit $(R, vu, uy).$ This has exactly two crossing chords, one is red (an edge in $G$ but non-edge in $G'$), the other is blue (a non-edge in $G$ but an edge in $G'$), and it alternates between red and blue chords: every second chord is red, the others are blue. Denote by $C^*$ the chord circuit derived from $C$ by switching blue and red chords along $C$ and denote by $G^*$ the graph derived from $G'$ by exchanging $C$ and $C^*.$ Then $G^*$ is clearly another realization of $\bd$ and since $C^*$ has two crossing chords, one is an edge, the other is a non-edge, therefore, $G^*$  is consistent with the skeleton graph $\cS(k).$
Since trail $R$ contains at least one red chord, the number of red edges in $E(G) \bigtriangleup E(G^*)$ is smaller than $r(G,G')$ and, consequently, the inductive assumption applies for $G$ and $G^*$: there exists a sequence $\Sigma_1$ of $\cS$-preserving swaps and double swaps transforming $G$ into $G^*.$ Furthermore, since $T$ was not closed, the number of red edges in $E(G') \bigtriangleup E(G^*)$ is smaller than $r(G,G')$ and, therefore, the inductive assumption applies for $G^*$ and $G'$ as well: there exists a sequence $\Sigma_2$ of $\cS$-preserving swaps and double swaps transforming $G^*$ into $G'.$ So the swap-sequence $\Sigma_1 \circ \Sigma_2$ transforms $G$ into $G'$ via $\cS$-preserving operations.

\medskip\noindent
If $uy$ is white, then we consider exactly the same circuit $C$ and let $C^*$ be derived in the same way. However, $G^*$ is now derived from $G$ by exchanging $C$ and $C^*.$
Again $G^*$ is consistent with the skeleton graph $\cS(k).$ Since $r(G,G^*) = |T| /2 < r(G,G')$, there exists an $\cS$-preserving swap sequence $\Sigma_1$ from $G$ to $G^*.$ Similarly there exists an $\cS$-preserving swap sequence $\Sigma_2$ transforming $G^*$ into $G'.$ So the swap-sequence $\Sigma_1 \circ \Sigma_2$ transforms $G$ into $G'$ via $\cS$-preserving operations.
\end{proof}

\medskip\noindent Now we are ready to finish the proof of Theorem \ref{thm:doubleswaps}: Consider the symmetric difference $\nabla$. Assume at first there is a connected component in $\nabla$ containing both red and blue crossing chords. Then they are connected with an alternating chord trail and traversing it will reveal a red and blue crossing chord pair as in Figure \ref{fig:base}. Therefore we may assume that no connected component has both red and blue crossing chords. Next we decompose it into alternating chord circuits and let $C_1$ and $C_2$ be two of them, the first has red crossing chords while the second one has blue crossing chords. By our assumptions these two circuits are vertex disjoint.

Let $\Sigma_1$ denote a regular swap sequence processing $C_1$ and let $\Sigma_2$ denote the analogous regular sequence for $C_2$. Now we execute $\Sigma_1$ step by step while the required swaps are $\cS$-preserving. Let $\Sigma_1(i)$ denote the last performed operation. If the swap $\Sigma_1(j)$, where $j\le i$, produced realization $H$ and $r(H,G')$ is smaller than $r(G,G')$, then the inductive hypothesis applies. So we suppose here that the operations $\Sigma_1(1),\ldots, \Sigma_1(i)$ only split the alternating chord circuits into smaller ones.

At this point we execute the swap sequence $\Sigma_2$ step by step while the required swaps are $\cS(k)$-preserving. Let $\Sigma_2(j)$ denote the last performed operation. We let~$H$ denote the current realization, which is consistent with $\cS(k)$ and has $r(H,G')$ equal to $r(G,G').$

\medskip\noindent
Now we execute the {\bf double swap} $\Sigma_1(i+1)$ and $\Sigma_2(j+1)$ which together produce a new realization $H'$ which is consistent with $\cS(k)$ and for which $r(H',G')$ is smaller than $r(G,G').$ The inductive hypothesis applies. This completes the proof of Theorem~\ref{thm:doubleswaps}. \qed

\section{Multipartite graph realizations}

This section considers skeleton graphs with more than two classes but without loops. First, we consider skeleton graphs that contain exactly one odd cycle.

\begin{lemma}\label{th:oddcycle}
Let $V$ be the underlying vertex set with partition $\Pi=(U_1,\ldots,$ $U_{2k-1})$, let $\bd$ be a sequence of~$|V|$ integers, and~$\cS=(\Pi;\cB,w)$ a skeleton graph consisting of exactly one odd cycle $U_1U_2, U_2U_3,\cdots ,U_{2k-1}U_1$ with an otherwise undefined weight function $w$. Then there exists at most one weight function $w$ for which $\G(\cS(w))$ is not empty. $($Here $\cS(w)$ is a shorthand for $\cS=(\Pi;\cB,w).)$
\end{lemma}
\begin{proof}
Assume that $G$ is a realization of $\bd$ which is weakly consistent with skeleton graph $\cS = (\Pi; \cB).$ (Recall: there is no edge in $G$ outside $\cB$.) Let $D(U)$ (for $U \in \Pi$) denote $D(U)=\sum_{u \in U}d(u).$ (This is the total degree of $U$ in the skeleton graph.) Furthermore let $\alpha$ denote the number of edges in the bone $U_1U_2$. Then there are $D(U_1) - \alpha$ edges in the bone $U_1U_{2k-1}.$

We can calculate the number of edges in the bone $U_2 U_3$: it is $D(U_2) - \alpha.$ The number of edges along the bone $U_3 U_4$ is $D(U_3)- D(U_2) + \alpha$.
And so on: we can calculate the number of edges in all bones, one by one. We finish it to calculate the number of edges in the bone $U_{2k-2}, U_{2k-1}$ which is some $\beta,$ a linear function of $\alpha$ with coefficient 1. Finally we know that $\alpha + \beta = D(U_{2k-1}).$ So $\alpha$ is fully determined. If a weakly consistent realizations exists then the solution for $\alpha$ must provide non-negative integer values $w^*(U_iU_{i+1})$ for all bones.
\end{proof}
\noindent
The easy consequence is that all weakly consistent realizations are consistent with exactly the same weight function: each belongs to $\cS(w^*)$.

By the above lemma, we can solve the existence problem for this particular skeleton graph by deciding if there exists a realization of~$\bd$ that is weakly consistent with~$\cS$. But Observation \ref{th:weak} does exactly this for us.

\begin{corollary}\label{th:corol}
There exists a polynomial algorithm to decide the existence of a consistent realization to the skeleton graph problem above.
\end{corollary}

\medskip
Before we proceed to even cycles we need some definitions and a result about restricted degree sequence problems from~\cite{restricted}. Let $F$ be a subset of the vertex pairs from $V$. The other vertex pairs on $V$ are called {\bf chords}. Let $\bd$ be a degree sequence on $V$.  We are interested in those realizations of $\bd$ which completely miss $F$ (the \emph{forbidden} set of non-chords). The set of all such realizations is denoted by $\G^F(\bd)$.

Consider a realization $G=(V,E) \in \G^F(\bd)$ where $v_1,\ldots ,v_{2k}$ is an alternating (edge, non-edge, ... etc.) circuit of chords. Assume that all pairs $v_iv_j$ which would divide $C$ into two even chord circuits (these are the pairs $i,j\in\{1, \ldots ,2k\}$ with $j=i + 1 \pmod{2}$ and $|i-j|>1$) are forbidden (they are not chords). Then the operation which exchanges the edges and non-edges along the circuit $C$ is called an {\bf $F$-swap}. It is clear that if $F=\emptyset$ then this notion coincides with Havel's swap notion. The following result follows directly from Theorem~2.3 of~\cite{restricted}: the space $\G^F(\bd)$ is connected under $F$-swaps. More precisely:
\begin{theorem}[\cite{restricted}]\label{th:Fswaps}
Let $G, G' \in \G^F(\bd)$ be two realizations. Then there exists a sequence of $F$-swaps which turns $G$ into $G'$. Moreover if the symmetric difference between $G$ and $G'$ is one alternating chord circuit, then all the $F$-swaps happen within the vertex set of $C.$ \qed
\end{theorem}

\bigskip\noindent
We continue our investigations with considering  skeleton graphs consisting of one even cycle: we use similar notations as before except that the last vertex partition is denoted by $U_{2k}$ and the cycle is modified accordingly.

The existence problem for such skeleton graphs was originally raised for the case $k=2$ by L\'aszl\'o A. Sz\'ekely~\cite{szekely}; he also suggested a solution for this particular question.

Now we can do the same calculation here that happened in the proof of Lemma \ref{th:oddcycle}. However, the final equation contains no $\alpha$, it is just an alternating sum of $D(U_i)$s, an identity. Therefore there may exist several feasible values for $\alpha,$ and it is possible to find all feasible weight functions (a constant number) and to find at least one consistent realization for each feasible weight function in polynomial time. More precisely:
\begin{lemma}\label{th:evencycle}
Let $\cS=(\Pi;\cB)$ be a skeleton graph consisting of exactly one cycle $C$ with an undefined weight function. Let $\alpha$ denote the weight of $U_1U_2$ under a realization consistent with~$\cS$. Then there exist a minimum possible value $\alpha_m$ and a maximum possible value $\alpha_M$, and each value $\alpha= \alpha_m,\alpha_m+1, \ldots, \alpha_M$ appears as feasible bone weight on $U_1U_2$. Finally one can provide one consistent realization for each possible weight function in polynomial time.
\end{lemma}
\begin{proof}
When our cycle $C$ has odd length, then Lemma \ref{th:oddcycle} and Corollary  \ref{th:corol} apply and we have nothing to prove. So assume now that $C=(U_1U_2,\ldots, U_{2k}U_1).$ Applying the method of Observation \ref{th:weak} with weight 0 for all chords in the  bone $U_1U_2$ and 1 for all other chords, the derived maximum weight perfect 1-factor provides the value $\alpha_m$ and a corresponding degree sequence realization. If we consider the opposite weight function then the maximum weight perfect 1-factor  provides the value $\alpha_M$ and a corresponding degree sequence realization.

Finally one can find in polynomial time at least one realization for each value $\alpha=\alpha_m,\alpha_m+1,\ldots, \alpha_M$ applying Theorem \ref{th:Fswaps} as follows: from this statement we know that $G$ can be transformed into $G'$ using $F$-swaps. In this setup the chords are the vertex pairs within the bones, all other vertex pairs are forbidden. When $G$ and $G'$ are from $\G(\Pi, \cB)$ then they can be consistent with different weight functions $w$ and $w'.$  In any procedure transforming $G$ into $G'$ each $F$-swap alters the edges and non-edges along an alternating chord circuit $C'$. There are two possibilities: this $C'$ can go around the bone-circuit $C$ zero or an even number of times or an odd number of times. In the first case for each bone the number of edges in this bone before and after the swap will be the same. In the second case the numbers of the edges within the bones increase and decrease with exactly one, alternately. Therefore each possible value $\alpha$ between $\alpha_m$ and $\alpha_M$ must occur in the bone $U_1U_2.$
\end{proof}

\bigskip\noindent
Now we discuss in short another possible multipartite skeleton graph:
\begin{lemma}\label{th:tree}
Let the skeleton graph $\cS=(\Pi; \cB, w)$ be a tree and let $L$ be a leaf in this tree. Assume that realization $G \in \G(\bd)$ is consistent with $\cS.$ Then the weight function $w$ and the value $D(L)$ are completely determined by the values $D(U),\ U\ne L.$
\end{lemma}
\begin{proof}
This statement is almost trivial. One can argue in the same way as it happened in the proof of Lemma \ref{th:oddcycle}: starting from the leaves different from $L$ and working along the paths toward $L$ one can determine all $D(UW)$ values along the tree.
At the last step on the unique bone $UL$ the value $w(UL)$ must be the same as $D(L).$
\end{proof}
Putting together these statements we have the following result:
\begin{theorem}\label{th:mixed}
Let $\bd$ be a sequence of $|V|$ positive integers and let $\cS=(\Pi;\cB)$ be a connected skeleton graph with at most one cycle. Then we can find in polynomial time all weight functions $w$ for which there exists realizations of $\bd$ which are consistent with the skeleton graph $\cS=(\Pi; \cB, w)$ along with at least one realization for each possible weight function.
\end{theorem}
\begin{proof}
We may and will assume that the skeleton graph has exactly one cycle $C$ and trees connected to the vertices of that cycle because otherwise Lemma \ref{th:tree} would apply. If at vertex class $U$ of $C$ in the connected tree the vertex $U$ has $a$ neighbors, then we consider $a$ disjoint subtrees, all rooted in $U$. For each subtree the application of Lemma \ref{th:tree} determine the corresponding weight function values, and together they determine the ``residual" $D(U)$ for the cycle $C$.  This can be done in polynomial time.

If $C$ is odd, then it provides one unique weight function as possible setup for consistent realizations. The usual application of the Tutte method provides in polynomial time a weakly consistent realization of $\bd$ which will be automatically consistent.

If $C$ is even, then the application of Lemma \ref{th:evencycle} provides the possible weight functions, together with actual consistent realizations for each possible weight function.
\end{proof}

\medskip\noindent
This finished the discussion of the existence problem of consistent realizations for skeleton graphs $\cS=(\Pi; \cB, w)$ where all connected components contain at most one cycle.  In the remainder of this section we discuss briefly the connectivity problem of the space of all consistent realizations.

First of all we have to recognize that instead of asking the connectivity of the realization space $\G(\cS)$ under the regular swap operations we have to consider the $F$-swap operations, defined by the forbidden edges outside the active bones.

Assume that our skeleton graph is connected and it has at most one cycle. If this cycle is odd, then there is at most one weight function for which $\G(\cS)$ is not empty, and every weakly consistent realization will be automatically consistent as well, so the $F$-swap operations do not destroy the consistency. The application of Theorem \ref{th:Fswaps} proves the connectivity of the space.

When the cycle under consideration is even then we have a more complex situation. First of all there may be several different weight functions with consistent realizations, and---similarly to the bipartite case $\cS(k)$---it is possible that $\G(\cS)$ is not connected  under $F$-swaps. However, again similarly to the bipartite case, one can organize the $F$-swap sequence such that whenever we have to leave the current weight function $w$ into $w'$---which differs from $w$ with one on each bone along the cycle---then the next $F$-swap goes back again to the original weight function. Thus, $\G(\cS)$ is connected under $F$-swaps and ``double $F$-swaps''.

\bibliographystyle{plain}

\end{document}